\documentclass{amsart}
\usepackage{amssymb, amsfonts, enumerate}
\usepackage{graphicx}

\newcommand{\Q}{\mathbb{Q}}
\newcommand{\N}{\mathbb{N}}

\newcommand{\C}{\mathbb{C}}

\newenvironment{Aut}{{\rm Aut}}{}

\renewcommand{\setminus}{\smallsetminus}

\newtheorem{thm}{Theorem}[section]
\newtheorem*{thm*}{Theorem}
\newtheorem{lem}[thm]{Lemma}
\newtheorem{cor}[thm]{Corollary}

\newtheorem*{conjecture*}{Conjecture}

\theoremstyle{remark}
\newtheorem{question}[thm]{Question}
\newtheorem{remark}[thm]{Remark}
\newtheorem{example}[thm]{Example}

\theoremstyle{definition}
\newtheorem{define}[thm]{Definition}

\begin{document}

\title[Blocks of monodromy groups]{Blocks of monodromy groups in Complex Dynamics}

\thanks{MSC class numbers: 37F10 (primary), 20B99 (secondary)}
\thanks{The first author's research was partially supported by NSF grant DMS-0852826.  The second author's research was partially supported by NSF grant DMS-0757856.}

\author[Jones]{Rafe Jones}
\address{
Department of Mathematics and CS,
College of the Holy Cross,
Worcester, MA 
}

\author[Peters]{Han Peters}
\address{
Department of Mathematics,
University of Amsterdam,
Amsterdam, Netherlands 
}
\email{peters@math.sunysb.edu}

\email{rjones@holycross.edu}



\begin{abstract}
Motivated by a problem in complex dynamics, we examine the block structure of the natural action of iterated monodromy groups on the tree of preimages of a generic point.  We show that in many cases, including when the polynomial has prime power degree, there are no large blocks other than those arising naturally from the tree structure.  However, using a method of construction based on real graphs of polynomials, we exhibit a non-trivial example of a degree $6$ polynomial failing to have this property. This example settles a problem raised in a recent paper of the second author regarding constant weighted sums of polynomials in the complex plane.  We also show that degree $6$ is exceptional in another regard, as it is the lowest degree for which the monodromy group of a polynomial is not determined by the combinatorics of the post-critical set.  These results give new applications of iterated monodromy groups to complex dynamics.
\end{abstract}

\maketitle


\section{Introduction}
\label{Introduction}

Let $G$ be a subgroup of the permutation group of a finite set $S$. A subset $E \subset S$ is called a {\emph{$G$-block}} if for every permutation $\sigma \in G$ we have that $\sigma(E) \cap E$ is equal to either $E$ or $\emptyset$. Motivated by a problem in a recent paper \cite{han} of the second author we study $G$-blocks for monodromy groups of polynomial iterates.  These groups are natural finite quotients of the iterated monodromy group of the given polynomial.  The latter have attracted attention mainly for their group-theoretic interest (see e.g. \cite{bartholdi}, \cite{bux}, \cite{nekcantor}, \cite{kai}), and to a lesser extent for the light they shed on the dynamics of the given polynomial (e.g. \cite[Chapter 6]{nekrashevych}).  The results of the present paper give a new application of iterated monodromy groups to complex dynamics, as well as results on their structure.  We remark that little remains known in general about iterated monodromy groups; see \cite{nekpoly} for some general discussion and open questions.

The monodromy group $MG_n(f)$ of the $n$th iterate $f^{\circ n}$ of a complex polynomial $f$ acts by definition on the set $R_n$ of $n$th preimages of a generic point.  Indeed, this action extends to
one on the tree of preimages of all iterates up to $n$.  The primary goal of the present paper is to show that for many, indeed in some sense most, $f$ there are no large $MG_n(f)$-blocks other than those coming from the tree structure.
Here is a special case of our main results in this vein (Theorems \ref{main1} and \ref{main3}):
\begin{thm}
Let $f \in \C[x]$ have degree $d \geq 2$, $f^{\circ n}$ be the $n$th iterate of $f$, and let the monodromy group $MG_n(f)$ of $f^{\circ n}$ act naturally on set $R_n$ of $n$th preimages of a generic point.  Suppose that either $d$ is prime or $MG(f)$ acts primitively (i.e. without non-trivial blocks) on $R_1$.  If $E \subseteq R_n$ is any set containing $a, b$ such that $f^{\circ n-1}(a) \neq f^{\circ n-1}(b)$, then the smallest $MG_n(f)$-block of $R_n$ containing $E$ is $R_n$ itself.
\end{thm}
We also prove a result on the structure of $MG_n(f)$-blocks in the case where $d$ is a power of a prime (Theorem \ref{main1}). 

Our other primary goal is to find polynomials $f$ whose monodromy groups have exceptional properties, such as $R_n$ containing small $MG_n(f)$-blocks not coming from the tree structure.  To do this, we give a method for constructing examples of critically finite $f$ with specified critical portrait (and often even specified monodromy groups of all iterates).  The method relies on building real graphs with certain properties, and gives polynomials with real coefficients.  For instance, in Theorem \ref{main2} we give a polynomial $h$ of degree $6$ such that $R_2$ contains an $MG_2(f)$-block of size 4 containing $a$ and $b$ with $h(a) \neq h(b)$.

A more detailed outline of the paper is as follows.  In Section \ref{postcrit} we discuss how the monodromy group of a polynomial $f$ depends on the critical portrait of $f$, that is, the orbits $c \to f(c) \to f^{\circ 2}(c) \to \cdots$ of the critical points of $f$. There we only look at the monodromy group of the first iterate and we show (Theorem \ref{portrait1} and Theorem \ref{portrait2}) that $6$ is the smallest degree for which the critical portrait of a polynomial does not determine the monodromy group up to conjugacy. In fact we give two explicit polynomials of degree $6$ with real coefficients and identical critical orbit but whose monodromy groups are not isomorphic.  The method of construction is used also in Section \ref{gblocks}.   In Section \ref{tree} we fix notation and give a few definitions.  

In Section \ref{gblocks} we prove the main results of this paper. We first show (Theorem \ref{main1}) that if $f$ is a polynomial whose degree is prime then the only $MG_n(f)$-blocks are branches, and when the degree of $f$ is a power of a prime then the only $MG_n(f)$-blocks are unions of equal-height branches, contained in a branch of height one more. Motivated by this result we call such blocks {\emph {basic blocks}}; see Definition \ref{basicblock}.  We then show (Theorem \ref{main3}) that if $MG_1(f)$ acts primitively on $R_1$, then any $MG_n(f)$-block containing elements that map to distinct members of $R_1$ must be all of $R_n$.  We apply this to the case where $f$ is a conservative polynomial, i.e. all critical points are also fixed points, and $f$ has at least two critical points (Corollary \ref{blockcor}).

Then we show (Theorem \ref{main2}) that there exist polynomials of non-prime-power degree for which there are blocks that are not basic blocks. In fact, we give an explicit example of a degree $6$ polynomial which has $MG_2(f)$-blocks that are non-basic. This polynomial is quite similar to the examples used in the proof of Theorem \ref{portrait2}.

In general it can be quite complicated to compute (iterated) monodromy groups of polynomials. However the polynomial examples presented in this paper have two properties in common that make it much easier to compute their monodromy groups. Firstly the post-critical sets are finite, which means that only a finite number of generators have to be considered. Secondly, the post-critical sets are real, which means that the action of the generators can be determined by looking only at the real graph of the polynomials.

Let us finish the introduction by discussing the motivation for this paper. In \cite{han} the following definition was introduced:

\begin{define}\label{cwsi}
A polynomial $f$ has a {\emph{constant weighted sum of iterates}} (c.w.s.i.) near $z \in \C$ if there exist weights $a_0, a_1, \ldots \in \C$ and a constant $c \in \C$ such that the finite sums
$$
\sum_{n=0}^{N} a_n f^{\circ n}
$$
converge uniformly to the constant function $c$ in a neighborhood of $z$ as $N \rightarrow \infty$.
\end{define}

The requirement that the maps $\sum_{n=0}^{N} a_n f^{\circ n}$ converge to the same constant for any point in the neighborhood of $z$ is very strong and one would not expect convergence to a constant to be possible for many values of $z$. The only cases in which it is known that $f$ has a constant weighted sum of iterates near $z$ is when $f$ is affine or when $z$ lies in a Siegel disc of $f$. Indeed, it was shown in \cite{han} that a generic polynomial $f$ of degree at least $2$ has a c.w.s.i. near $z$ if and only if $z$ lies in a Siegel disc of $f$.

The proof of this result relies on the fact that the monodromy groups of the iterates of a generic polynomial (a polynomial whose iterates $f^{\circ n}$ each have the maximal number of distinct critical values) are equal to the entire group of tree automorphisms. In general the monodromy groups of the iterates may be much smaller but in many cases the monodromy groups can still be used to prove the same result. In fact, if the following question can be answered affirmatively then it still follows that a polynomial of degree at least $2$ has a c.w.s.i near $z$ if and only if $z$ lies in a Siegel disc:

\begin{question}\label{fatou}
Let $f$ be a polynomial of degree at least $2$ and let $V$ be a bounded Fatou component of $f$ that is eventually mapped onto a periodic (super-) attracting basin or attracting petal. Let $N \in \N$.

In the case of an attracting basin or attracting petal, let $p \in F^{\circ N}(V)$ have $d^N$ distinct pre-images and define the sets $S = f^{\circ-N}(\{p\})$ and $E = S\cap V$. In the case of a super-attracting basin, let $U \subset f^{\circ N}(V)$ be an arbitrarily small neighborhood of the super-attracting periodic point in $f^{\circ N}(V)$. Let $p \in U$ again have $d^N$ distinct pre-images and define $S = f^{\circ-N}({p})$ as before. Now $E \subset S$ contains only points that lie in one chosen connected component of $f^{\circ -N}(U)$ in $V$.

Is it possible to choose $N$ such that the smallest $MG(f^{\circ N})$-block of $S$ that contains $E$ has more than $d^{N-1}$ elements?
\end{question}

The orbit of a periodic Fatou component that is a (super-) attracting basin or an attracting petal must always contain a critical point. The integer $N$ can be chosen such that $f^N(V)$ contains the corresponding critical value. As was explained in \cite{han}, we may assume that the neighborhood $U$ of $z$ in Definition \ref{cwsi} is large enough so that $f^{\circ N}(U)$ contains the critical value. It follows that when $p$ is chosen close enough to this critical value, the set $E$ will contain two points $e_1, e_2$ with the property that $f^{\circ N - 1}(e_1) \neq f^{\circ N - 1}(e_2)$.  In the terminology of section \ref{tree}, the points $e_1$ and $e_2$ belong to different major branches.

Theorems \ref{main1} and \ref{main3} say that we can answer Question \ref{fatou} affirmatively when the degree of $f$ is a power of a prime number or when $MG(f)$ acts primitively on the roots of $f$, but negatively for certain polynomials of degree $6$. We obtain the following corollary.

\begin{cor}\label{cor}
Suppose that $f$ is a polynomial with complex coefficients, and that either $f$ has degree $p^k$ with $p$ prime, or $MG(f)$ acts primitively on the $f$-preimages of a generic point.  Then $f$ has a constant weighted sum of iterates near $z \in \C$ if and only if $z$ is eventually mapped into a Siegel disc of $f$.
\end{cor}

The negative answer to Question \ref{fatou} only closes off one avenue of proof when $f$ has degree $6$.  It does not imply that it is possible to have a constant weighted sum of iterates when $z$ does not lie in a Siegel disc.  Whether this can occur for a polynomial of degree at least $2$ is still open.

\section{Combinatorics of the post-critical set} \label{postcrit}

Let $f$ be a polynomial of degree $d$.  Denote by $C$ the set of critical points of $f$, and $V$ the set of critical values of $f$, that is, $V = \{f(c) : c \in C\}$.
Then $f : \C \setminus C \to \C \setminus V$ is a covering, and if $p \in \C$ has $d$ distinct inverse images, then every closed loop at $p$ contained in $\C \setminus V$ induces a permutation on the inverse images of $p$.  The subgroup of $S_d$ thus obtained is independent of the point $p$ and is called the {\em monodromy group} of $f$, denoted $MG(f)$.  Let $f^{\circ n}$ denote the $n$th iterate of $f$ (i.e. the $n$-fold composition of $f$ with itself).   We denote by $MG_n(f)$ the monodromy group of $f^{\circ n}$.  It is not a general subgroup of $S_{d^n}$, as it must respect the natural tree structure on preimages of $p$ under $f$ (see Section \ref{tree} for details).

The critical portrait (namely the number and multiplicity of the critical values) of a polynomial $f$ gives us a large amount of information about the monodromy groups of $f^{\circ n}$. For example, let $v$ be a critical value of a polynomial $f$ of degree $d$, and denote the inverse images of $v$ by $z_1, \ldots z_n$. A small enough neighborhood $U$ of $v$ will have $n$ disconnected inverse images $V_1, \ldots, V_n$ with $z_j \in v_j$. Let $\gamma \in U$ be a closed loop at $p \in U$ given by a circle centered at $v$. Then the inverse images of $\gamma$ lie in the sets $V_1, \ldots, V_n$. If $z_j$ is a critical point of $f$ with order $k$ then $\gamma$ induces a full cycle on the $k$ preimages of $p$ in $V_j$.

\label{cycle} If $\gamma$ is a closed loop given by a circle centered at the origin of large enough radius then it induces a full cycle. Indeed, if the radius of $\gamma$ is larger than the modulus of all the (finite) critical points then we can view it as a loop around the point at infinity, a critical point of order equal to the degree of the map.

\begin{thm}\label{portrait1}
The monodromy group of a polynomial of degree at most $5$ is completely determined by its critical portrait, that is, the number and multiplicity of its critical values.
\end{thm}
\begin{proof}
This is merely a case of checking all possibilities. We will do so for degree $5$ and leave the smaller degrees to the reader.   

Since the degree is prime and $MG(f)$ is transitive, it must be a primitive subgroup of $S_5$.  It is well known (see e.g. \cite{wielandt}) that if $G$ is a primitive subgroup of $S_n$ that contains a transposition, then $G = S_n$, while if $G$ contains a 3-cycle, then $G$ contains $A_n$.  Thus if $f$ has a critical value that is the image of a single critical point of multiplicity 2, then $MG(f) = S_n$.  Otherwise, by the Riemann-Hurwitz formula, $f$ must have either one or two critical values.  In the former case, $MG(f)$ has only one generator and thus $MG(f) \cong C_5$.  In the latter case, each critical value must either be the image of one critical point of multiplicity 3 or two critical points of multiplicity 2.  Thus $MG(f)$ may be generated by (i) two $3$-cycles, (ii) one $3$-cycle and one element of type $(2,2)$ or (iii) two elements of type $(2,2)$.  In all three cases we have $MG(f) \subseteq A_5$, and for $(i)$ and $(ii)$ we have that $MG(f)$ contains a 3-cycle, and so $MG(f) = A_5$.  In case (iii) after conjugation we may assume the generators are $(12)(34)$ and $(23)(45)$ and we obtain $MG(f) \cong D_5$.  
\end{proof}

In degree $6$ the situation is more complicated:

\begin{thm}\label{portrait2}
There exist polynomials $f, g$ of degree $6$ with identical critical portraits but different monodromy groups. Moreover, one can choose $f$ and $g$ to have real coefficients.
\end{thm}

\begin{remark}
One can provide a short proof of Theorem \ref{portrait2} using the Riemann Existence Theorem (see \cite{debes} for an exposition), which guarantees that there is $f \in \C[z]$ of degree 6 with $MG(f) = \langle (23)(45), (34), (12)(56) \rangle \neq S_6$, and $g \in \C[z]$ of degree 6, with identical critical portrait to $f$ and $MG(g) = \langle (23)(45), (12), (34)(56) \rangle = S_6$.  However, the Riemann Existence Theorem is non-constructive, and for our purposes it will be advantageous to have an explicit method to find such $f$ and $g$, since in Theorem \ref{main2} we need to construct a polynomial whose second iterate has prescribed monodromy.  In addition, the Riemann Existence Theorem does not guarantee that $f$ and $g$ can be chosen to have real coefficients.  
\end{remark}

\begin{proof} We construct $f$ and $g$ explicitly.  Both have 5 distinct critical points that all lie on the real line.  By choosing our base point also on the real line and choosing our generator loops close to the real line we can use the real graphs of $f$ and $g$ to determine the monodromy groups.
For both $f$ and $g$, two critical points map to a third critical point, and this third critical point gets mapped to a fourth critical point which is a fixed point. The fifth critical point is also fixed. Therefore the critical portraits of $f$ and $g$ are identical.

\begin{figure}[htbp]
\begin{center}
\includegraphics[width = 4in]{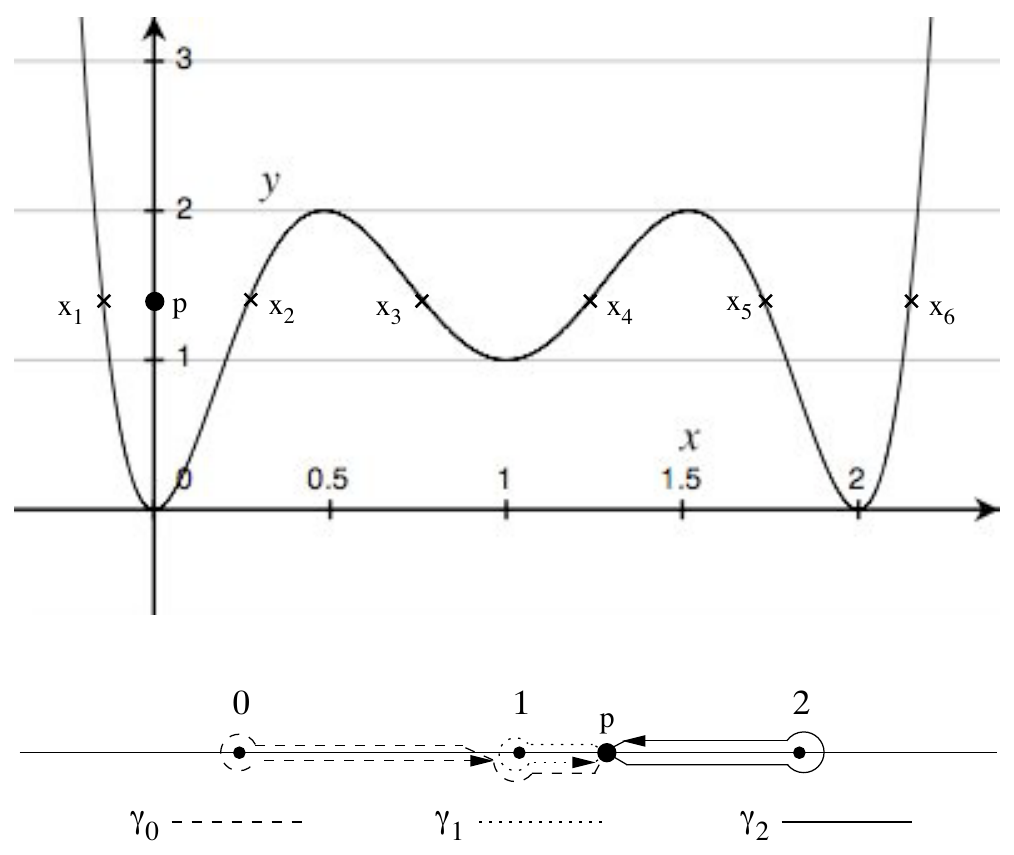}
\end{center}
\caption{Real graph of $f$ and generators of the corresponding fundamental group.}
\label{fig: graph}
\end{figure}

We now describe $f$.  Its critical points are $0, c_1, 1, c_2$ and $2$, where $0 < c_1 < 1 < c_2 < 2$,
$f(c_1) = f(c_2)= 2$, $f(0) = f(2) = 0$ and $f(1) = 1$. This guarantees that $f$ has the critical portrait described above. See Figure \ref{fig: graph} for the real graph of $f$.  The existence of $f$ can be proved by checking the degrees of freedom, but in fact a formula for $f$ can easily be found as follows.  Start with an even polynomial $h$ that has critical points at $0$, $\pm 1$ and $\pm \alpha$, where $h(0) = 0$ and $h(1) = -h(\alpha)$ (taking $\alpha = \sqrt{(2+\sqrt{3})}$ works). The function $f$ is then given by $f(z) = 1 + \frac{h(\alpha(z-1))}{h(1)}$; note that $c_1 = 1-1/\alpha$, $c_2 = 1 + 1/\alpha$.

The critical points of $g$ are $c_1, c_2, 0, c_3$ and $1$, where $c_1 < c_2 < 0 < c_3 < 1$. Now we require that $g(c_1) = c_1$, $g(c_2) = g(c_3) = 1$ and $g(1) = g(0) = 0$.  Again the existence of $g$ can be proved by counting degrees of freedom.

Note that $f$ is the composition of two polynomials of degree $3$ and $2$ respectively. Therefore one immediately sees that the monodromy group of $f$ must have three blocks with $2$ elements each.  Once we determine the monodromy group of $g$, which has no non-trivial blocks, this will be enough to prove the Theorem.  However, since the computations of $MG(f)$ and $MG(g)$ are very similar, we will discuss them both.  

To determine the monodromy group of $f$ we choose a base point $1 < p < 2$ and we denote its (real) preimages by $x_1 < \cdots < x_6$. The critical values of $f$ are $0, 1,$ and $2$, and we define the corresponding generating loops $\gamma_0, \gamma_1,$ and $\gamma_2$ as follows (see Figure \ref{fig: graph}).  The loop $\gamma_2$ moves along the real axis from $p$ to $2-\epsilon$, for $\epsilon>0$ very small, then follows a full clockwise circle centered at $2$ and goes back to $p$ along the real axis. When following the pre-images of this loop the points $x_2$ and $x_3$ first move very close to $c_1$ (thus moving upwards on the real graph of $f$), then switch and move back to $x_3$ and $x_2$, respectively. Similarly, $x_4$ and $x_5$ go very close to $c_2$, then switch and move back to respectively $x_5$ and $x_4$. The preimages of $\gamma_2$ starting at $x_1$ and $x_6$ form closed loops.  Hence assigning to $x_i$ the number $i$, we have that $\gamma_2$ induces $(23)(45) \in S_6$.

The loop $\gamma_1$ is similar: starting at $p$ follow the real axis to $1 + \epsilon$, then loop around $1$ and follow the real axis back to $p$.   By checking the real graph of $f$ one sees that $\gamma_1$ induces $(34)$.  Finally, the loop $\gamma_0$ starts at $p$, follows the real axis to $1 + \epsilon$, follows a small semi-circle centered at $1$ and contained in the lower half plane, then follows the real axis again until very close to $0$, follows a full circle around $0$ and then takes the same path back to $p$.  The graph of $f$ shows that $\gamma_0$ induces $(12)(56)$.

Thus the monodromy group of $f$ is generated by the elements $(23)(45), (34),$ and $(12)(56)$. Notice that the sets $E_1 =\{1,6\}$, $E_2=\{2, 5\}$ and $E_3=\{3, 4\}$ form a partition of ${1, 2, 3, 4, 5, 6}$ that is invariant under the action $\gamma_0, \gamma_1,$ and $\gamma_2$, and hence for all of $MG(f)$.  In other words $E_1, E_2$ and $E_3$ are $MG(f)$-blocks and $MG(f) \neq S_6$.

For the monodromy group of $g$ we take a base point $0< p < 1$ and we denote its preimages by $x_1 < \cdots < x_6$ as before. The critical values of $g$ are $0, 1$ and $c:= c_1$, so we define loops $\gamma_{c}, \gamma_0$ and $\gamma_1$ again following the real axis as much as possible as for $f$, and we obtain that $\gamma_0$ induces $(34)(56)$, $\gamma_1$ induces $(23)(45)$ and $\gamma_c$ induces $(12)$. Since the elements $(34)(56), (23)(45)$ and $(12)$ generate $S_6$ we get $MG(g) = S_6 \neq MG(f)$ which completes the proof.
\end{proof}

\section{Tree structure} \label{tree}

The inverse images of $p$ under $f$ have a natural tree structure.  Indeed, let $R_n = \{x \in \C : f^{\circ n}(x) = p\}$, and note that
\[
R_n = \bigsqcup_{x \in R_{n-1}} f^{-1}(x).
\]
Thus $R_n$ may be thought of as the top level of a tree $T_n$, where each vertex $x$ is connected to
$f(x)$ in the next level down, namely $R_{n-1}$.  Similarly, each element of $R_{n-1}$ is connected to an element of $R_{n-2}$, and so on down to the level $0$, which consists only of $p$.  Note that $p$ connects to exactly $d$ vertices, vertices in the top level connect to only one other vertex, and all vertices besides $p$ and those not in the top level connect to $d + 1$ other vertices.  Setting $R_1 = \{x_1, \ldots, x_d\}$, we refer to the $d$ sets $\{f^{-(n-1)}(x_1)\}, \ldots, \{f^{-(n-1)}(x_d)\}$ as the {\em major branches} of $R_n$.  Note that $R_n$ is the disjoint  union of the major branches.   We call a {\em branch} of $R_n$ any subset of the form $f^{-m}(x)$, where $x \in R_{n-m}$.  We call $m$ the {\em height} of such a branch.  Hence the major branches are simply the branches of height $n-1$ and we consider the entire tree a branch of height $n$.

By definition, the group $MG_n(f)$ acts on $R_n$, and indeed we may think of it as acting on $T_n$ since any $\gamma \in \pi_1(\C \setminus \{\text{critical values of $f^{\circ n}$}\})$ may be lifted to any vertex of the tree.   A key feature of this action is that it preserves the tree structure of $T_n$.   Indeed, $x, y \in T_n$ are connected if and only if $y = f(x)$ or $x = f(y)$; assume without loss of generality that $x = f(y)$.   If $\gamma_y$ is the lift of $\gamma$ under $f^{\circ n}$ that starts at $y$, then $f(\gamma_y)$ is a (and hence the) lift of $\gamma$ under $f^{\circ n-1}$ starting at $x$.  Since $\gamma(y)$ is the endpoint of $\gamma_y$ and $\gamma(x)$ is the endpoint of $f(\gamma_y)$, we have $\gamma(x) = f(\gamma(y))$, implying that $\gamma(x)$ and $\gamma(y)$ are connected.  Thus the tree structure is preserved, giving an injection $MG_n(f) \hookrightarrow \Aut(T_n)$.

The preservation of the tree structure of $T_n$ immediately gives that every branch of $R_n$ is an $MG_n(f)$-block.  It can also occur that an $MG_n(f)$-block is not a single branch but a union of branches of equal height, contained in a single branch of height one more. The easiest polynomial for which this can occur is $f:z \mapsto z^4$. Then we have that $MG(f)$ is a cyclic group of order $4$, so this group has two blocks containing two elements. From now on we will refer to such blocks as basic blocks:

\begin{define} \label{basicblock}
Let $G$ be a subgroup of $\Aut(T_n)$ and let $E \subset R_n$ be a $G$-block. If $E$ is either a single branch or is a union of equal-height branches that are all contained in a single branch of height one more, then we say that $E$ is a {\em basic} block.
\end{define}

\begin{figure}

\hspace{-0.4 in}
\def\JPicScale{0.8}
\ifx\JPicScale\undefined\def\JPicScale{1}\fi
\unitlength \JPicScale mm
\begin{picture}(149.53,50)(0,0)
\linethickness{0.3mm}
\put(40.26,2.63){\line(0,1){20.59}}
\linethickness{0.3mm}
\multiput(9.58,38.65)(0.12,-0.24){64}{\line(0,-1){0.24}}
\linethickness{0.3mm}
\put(17.25,23.21){\line(0,1){15.44}}
\linethickness{0.3mm}
\multiput(17.25,23.21)(0.12,0.24){64}{\line(0,1){0.24}}
\linethickness{0.3mm}
\multiput(32.59,38.65)(0.12,-0.24){64}{\line(0,-1){0.24}}
\linethickness{0.3mm}
\put(40.26,23.21){\line(0,1){15.44}}
\linethickness{0.3mm}
\multiput(40.26,23.21)(0.12,0.24){64}{\line(0,1){0.24}}
\linethickness{0.3mm}
\multiput(55.59,38.65)(0.12,-0.24){64}{\line(0,-1){0.24}}
\linethickness{0.3mm}
\put(63.26,23.21){\line(0,1){15.44}}
\linethickness{0.3mm}
\multiput(63.26,23.21)(0.12,0.24){64}{\line(0,1){0.24}}
\linethickness{0.3mm}
\put(9.58,38.65){\line(0,1){10.29}}
\linethickness{0.3mm}
\multiput(6.51,48.95)(0.12,-0.4){26}{\line(0,-1){0.4}}
\linethickness{0.3mm}
\multiput(9.58,38.65)(0.12,0.4){26}{\line(0,1){0.4}}
\linethickness{0.3mm}
\multiput(14.18,48.95)(0.12,-0.4){26}{\line(0,-1){0.4}}
\linethickness{0.3mm}
\put(17.25,38.65){\line(0,1){10.29}}
\linethickness{0.3mm}
\multiput(17.25,38.65)(0.12,0.4){26}{\line(0,1){0.4}}
\linethickness{0.3mm}
\multiput(40.26,2.63)(0.13,0.117){177}{\line(1,0){0.13}}
\linethickness{0.3mm}
\multiput(17.25,23.21)(0.13,-0.117){176}{\line(1,0){0.13}}
\linethickness{0.3mm}
\put(24.92,38.65){\line(0,1){10.29}}
\linethickness{0.3mm}
\multiput(21.85,48.95)(0.12,-0.4){26}{\line(0,-1){0.4}}
\linethickness{0.3mm}
\multiput(24.92,38.65)(0.12,0.4){26}{\line(0,1){0.4}}
\linethickness{0.3mm}
\multiput(29.52,48.95)(0.12,-0.4){26}{\line(0,-1){0.4}}
\linethickness{0.3mm}
\put(32.59,38.65){\line(0,1){10.29}}
\linethickness{0.3mm}
\multiput(32.59,38.65)(0.12,0.4){26}{\line(0,1){0.4}}
\linethickness{0.3mm}
\put(40.26,38.65){\line(0,1){10.29}}
\linethickness{0.3mm}
\multiput(37.19,48.95)(0.12,-0.4){26}{\line(0,-1){0.4}}
\linethickness{0.3mm}
\multiput(40.26,38.65)(0.12,0.4){26}{\line(0,1){0.4}}
\linethickness{0.3mm}
\multiput(44.86,48.95)(0.12,-0.4){26}{\line(0,-1){0.4}}
\linethickness{0.3mm}
\put(47.93,38.65){\line(0,1){10.29}}
\linethickness{0.3mm}
\multiput(47.93,38.65)(0.12,0.4){26}{\line(0,1){0.4}}
\linethickness{0.3mm}
\multiput(52.53,48.95)(0.12,-0.4){26}{\line(0,-1){0.4}}
\linethickness{0.3mm}
\put(55.59,38.65){\line(0,1){10.29}}
\linethickness{0.3mm}
\multiput(55.59,38.65)(0.12,0.4){26}{\line(0,1){0.4}}
\linethickness{0.3mm}
\multiput(60.2,48.95)(0.12,-0.4){26}{\line(0,-1){0.4}}
\linethickness{0.3mm}
\put(63.26,38.65){\line(0,1){10.29}}
\linethickness{0.3mm}
\multiput(63.26,38.65)(0.12,0.4){26}{\line(0,1){0.4}}
\linethickness{0.3mm}
\multiput(67.86,48.95)(0.12,-0.4){26}{\line(0,-1){0.4}}
\linethickness{0.3mm}
\put(70.93,38.65){\line(0,1){10.29}}
\linethickness{0.3mm}
\multiput(70.93,38.65)(0.12,0.4){26}{\line(0,1){0.4}}
\linethickness{0.15mm}
\qbezier(75.53,48.95)(75.45,48.51)(73.18,48.22)
\qbezier(73.18,48.22)(70.9,47.93)(67.48,47.92)
\qbezier(67.48,47.92)(64.06,47.93)(61.79,48.22)
\qbezier(61.79,48.22)(59.52,48.51)(59.43,48.95)
\qbezier(59.43,48.95)(59.52,49.39)(61.79,49.68)
\qbezier(61.79,49.68)(64.06,49.97)(67.48,49.98)
\qbezier(67.48,49.98)(70.9,49.97)(73.18,49.68)
\qbezier(73.18,49.68)(75.45,49.39)(75.53,48.95)
\linethickness{0.3mm}
\put(114.26,2.63){\line(0,1){20.59}}
\linethickness{0.3mm}
\multiput(83.58,38.65)(0.12,-0.24){64}{\line(0,-1){0.24}}
\linethickness{0.3mm}
\put(91.25,23.21){\line(0,1){15.44}}
\linethickness{0.3mm}
\multiput(91.25,23.21)(0.12,0.24){64}{\line(0,1){0.24}}
\linethickness{0.3mm}
\multiput(106.59,38.65)(0.12,-0.24){64}{\line(0,-1){0.24}}
\linethickness{0.3mm}
\put(114.26,23.21){\line(0,1){15.44}}
\linethickness{0.3mm}
\multiput(114.26,23.21)(0.12,0.24){64}{\line(0,1){0.24}}
\linethickness{0.3mm}
\multiput(129.59,38.65)(0.12,-0.24){64}{\line(0,-1){0.24}}
\linethickness{0.3mm}
\put(137.26,23.21){\line(0,1){15.44}}
\linethickness{0.3mm}
\multiput(137.26,23.21)(0.12,0.24){64}{\line(0,1){0.24}}
\linethickness{0.3mm}
\put(83.58,38.65){\line(0,1){10.29}}
\linethickness{0.3mm}
\multiput(80.51,48.95)(0.12,-0.4){26}{\line(0,-1){0.4}}
\linethickness{0.3mm}
\multiput(83.58,38.65)(0.12,0.4){26}{\line(0,1){0.4}}
\linethickness{0.3mm}
\multiput(88.18,48.95)(0.12,-0.4){26}{\line(0,-1){0.4}}
\linethickness{0.3mm}
\put(91.25,38.65){\line(0,1){10.29}}
\linethickness{0.3mm}
\multiput(91.25,38.65)(0.12,0.4){26}{\line(0,1){0.4}}
\linethickness{0.3mm}
\multiput(114.26,2.63)(0.13,0.117){177}{\line(1,0){0.13}}
\linethickness{0.3mm}
\multiput(91.25,23.21)(0.13,-0.117){176}{\line(1,0){0.13}}
\linethickness{0.3mm}
\put(98.92,38.65){\line(0,1){10.29}}
\linethickness{0.3mm}
\multiput(95.85,48.95)(0.12,-0.4){26}{\line(0,-1){0.4}}
\linethickness{0.3mm}
\multiput(98.92,38.65)(0.12,0.4){26}{\line(0,1){0.4}}
\linethickness{0.3mm}
\multiput(103.52,48.95)(0.12,-0.4){26}{\line(0,-1){0.4}}
\linethickness{0.3mm}
\put(106.59,38.65){\line(0,1){10.29}}
\linethickness{0.3mm}
\multiput(106.59,38.65)(0.12,0.4){26}{\line(0,1){0.4}}
\linethickness{0.3mm}
\put(114.26,38.65){\line(0,1){10.29}}
\linethickness{0.3mm}
\multiput(111.19,48.95)(0.12,-0.4){26}{\line(0,-1){0.4}}
\linethickness{0.3mm}
\multiput(114.26,38.65)(0.12,0.4){26}{\line(0,1){0.4}}
\linethickness{0.3mm}
\multiput(118.86,48.95)(0.12,-0.4){26}{\line(0,-1){0.4}}
\linethickness{0.3mm}
\put(121.93,38.65){\line(0,1){10.29}}
\linethickness{0.3mm}
\multiput(121.93,38.65)(0.12,0.4){26}{\line(0,1){0.4}}
\linethickness{0.3mm}
\multiput(126.53,48.95)(0.12,-0.4){26}{\line(0,-1){0.4}}
\linethickness{0.3mm}
\put(129.59,38.65){\line(0,1){10.29}}
\linethickness{0.3mm}
\multiput(129.59,38.65)(0.12,0.4){26}{\line(0,1){0.4}}
\linethickness{0.3mm}
\multiput(134.2,48.95)(0.12,-0.4){26}{\line(0,-1){0.4}}
\linethickness{0.3mm}
\put(137.26,38.65){\line(0,1){10.29}}
\linethickness{0.3mm}
\multiput(137.26,38.65)(0.12,0.4){26}{\line(0,1){0.4}}
\linethickness{0.3mm}
\multiput(141.86,48.95)(0.12,-0.4){26}{\line(0,-1){0.4}}
\linethickness{0.3mm}
\put(144.93,38.65){\line(0,1){10.29}}
\linethickness{0.3mm}
\multiput(144.93,38.65)(0.12,0.4){26}{\line(0,1){0.4}}
\linethickness{0.15mm}
\qbezier(149.53,48.95)(149.45,48.51)(147.18,48.22)
\qbezier(147.18,48.22)(144.9,47.93)(141.48,47.92)
\qbezier(141.48,47.92)(138.06,47.93)(135.79,48.22)
\qbezier(135.79,48.22)(133.52,48.51)(133.43,48.95)
\qbezier(133.43,48.95)(133.52,49.39)(135.79,49.68)
\qbezier(135.79,49.68)(138.06,49.97)(141.48,49.98)
\qbezier(141.48,49.98)(144.9,49.97)(147.18,49.68)
\qbezier(147.18,49.68)(149.45,49.39)(149.53,48.95)
\linethickness{0.15mm}
\put(125.99,48.94){\line(0,1){0.12}}
\multiput(125.93,48.82)(0.06,0.12){1}{\line(0,1){0.12}}
\multiput(125.82,48.7)(0.12,0.12){1}{\line(0,1){0.12}}
\multiput(125.65,48.59)(0.17,0.11){1}{\line(1,0){0.17}}
\multiput(125.42,48.48)(0.22,0.11){1}{\line(1,0){0.22}}
\multiput(125.15,48.38)(0.27,0.1){1}{\line(1,0){0.27}}
\multiput(124.83,48.29)(0.32,0.09){1}{\line(1,0){0.32}}
\multiput(124.47,48.21)(0.36,0.08){1}{\line(1,0){0.36}}
\multiput(124.07,48.14)(0.4,0.07){1}{\line(1,0){0.4}}
\multiput(123.64,48.09)(0.43,0.06){1}{\line(1,0){0.43}}
\multiput(123.19,48.05)(0.45,0.04){1}{\line(1,0){0.45}}
\multiput(122.72,48.02)(0.47,0.03){1}{\line(1,0){0.47}}
\multiput(122.24,48)(0.48,0.01){1}{\line(1,0){0.48}}
\put(121.76,48){\line(1,0){0.48}}
\multiput(121.28,48.02)(0.48,-0.01){1}{\line(1,0){0.48}}
\multiput(120.81,48.05)(0.47,-0.03){1}{\line(1,0){0.47}}
\multiput(120.36,48.09)(0.45,-0.04){1}{\line(1,0){0.45}}
\multiput(119.93,48.14)(0.43,-0.06){1}{\line(1,0){0.43}}
\multiput(119.53,48.21)(0.4,-0.07){1}{\line(1,0){0.4}}
\multiput(119.17,48.29)(0.36,-0.08){1}{\line(1,0){0.36}}
\multiput(118.85,48.38)(0.32,-0.09){1}{\line(1,0){0.32}}
\multiput(118.58,48.48)(0.27,-0.1){1}{\line(1,0){0.27}}
\multiput(118.35,48.59)(0.22,-0.11){1}{\line(1,0){0.22}}
\multiput(118.18,48.7)(0.17,-0.11){1}{\line(1,0){0.17}}
\multiput(118.07,48.82)(0.12,-0.12){1}{\line(0,-1){0.12}}
\multiput(118.01,48.94)(0.06,-0.12){1}{\line(0,-1){0.12}}
\put(118.01,48.94){\line(0,1){0.12}}
\multiput(118.01,49.06)(0.06,0.12){1}{\line(0,1){0.12}}
\multiput(118.07,49.18)(0.12,0.12){1}{\line(0,1){0.12}}
\multiput(118.18,49.3)(0.17,0.11){1}{\line(1,0){0.17}}
\multiput(118.35,49.41)(0.22,0.11){1}{\line(1,0){0.22}}
\multiput(118.58,49.52)(0.27,0.1){1}{\line(1,0){0.27}}
\multiput(118.85,49.62)(0.32,0.09){1}{\line(1,0){0.32}}
\multiput(119.17,49.71)(0.36,0.08){1}{\line(1,0){0.36}}
\multiput(119.53,49.79)(0.4,0.07){1}{\line(1,0){0.4}}
\multiput(119.93,49.86)(0.43,0.06){1}{\line(1,0){0.43}}
\multiput(120.36,49.91)(0.45,0.04){1}{\line(1,0){0.45}}
\multiput(120.81,49.95)(0.47,0.03){1}{\line(1,0){0.47}}
\multiput(121.28,49.98)(0.48,0.01){1}{\line(1,0){0.48}}
\put(121.76,50){\line(1,0){0.48}}
\multiput(122.24,50)(0.48,-0.01){1}{\line(1,0){0.48}}
\multiput(122.72,49.98)(0.47,-0.03){1}{\line(1,0){0.47}}
\multiput(123.19,49.95)(0.45,-0.04){1}{\line(1,0){0.45}}
\multiput(123.64,49.91)(0.43,-0.06){1}{\line(1,0){0.43}}
\multiput(124.07,49.86)(0.4,-0.07){1}{\line(1,0){0.4}}
\multiput(124.47,49.79)(0.36,-0.08){1}{\line(1,0){0.36}}
\multiput(124.83,49.71)(0.32,-0.09){1}{\line(1,0){0.32}}
\multiput(125.15,49.62)(0.27,-0.1){1}{\line(1,0){0.27}}
\multiput(125.42,49.52)(0.22,-0.11){1}{\line(1,0){0.22}}
\multiput(125.65,49.41)(0.17,-0.11){1}{\line(1,0){0.17}}
\multiput(125.82,49.3)(0.12,-0.12){1}{\line(0,-1){0.12}}
\multiput(125.93,49.18)(0.06,-0.12){1}{\line(0,-1){0.12}}

\end{picture}

\caption{A basic block, left, consisting of branches of height $1$ contained in a single branch of height $2$.  Right, a non-basic block.}
\label{basic1}
\end{figure}
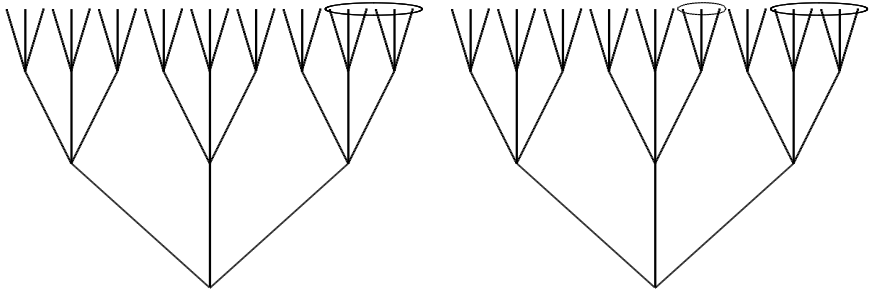

See Figure \ref{basic1} for an illustration of a basic block and a non-basic block.  The following example shows that there are subgroups $G < \Aut(T_n)$ with $G$-blocks that are not basic blocks.

\begin{example}\label{example1}
Let $G$ be the Klein 4-group acting on the complete binary tree of height $2$.  It is generated by the elements $(x_{00} \hspace{0.1 in} x_{01})(x_{10} \hspace{0.1 in}  x_{11})$ and $(x_{00} \hspace{0.1 in}  x_{10})(x_{01} \hspace{0.1 in}  x_{11})$. Then $G$ is a subgroup of $\Aut(T_2)$, yet $\{x_{00}, x_{10}\}$ and $\{x_{01}, x_{11}\}$ form an invariant partition.  Both of these are unions of height $0$ branches that are not contained in a height $1$ branch, and so are not basic.  
\end{example}

As noted in Section \ref{postcrit}, p. \pageref{cycle}, the monodromy group of a polynomial always contains a full cycle. Hence the group in Example \ref{example1} cannot occur as the monodromy group of a polynomial. In the next section we see that for many polynomials $f$ we only need the existence of a full cycle to deduce strong consequences for the possible structures of $MG_n(f)$-blocks.

\section{G-blocks} \label{gblocks}

In this section we prove results about automorphism groups of rooted trees (Theorems \ref{main1} and \ref{main3}), and then derive Corollary \ref{cor}.  We also give a negative answer to Question \ref{fatou} in Theorem \ref{main2}.  By slight abuse of notation we use the same notation ($T_n, R_n$, etc.) as the last section to refer to abstract trees.

\begin{thm}\label{main1}
Let $T_n$ be the complete $d$-ary rooted tree of height $n$, and let $R_n$ denote the vertices of level $n$.  Suppose that $G \leq \Aut(T_n)$ contains a full $d^n$-cycle $\sigma$.  If $d$ is a prime power then every $G$-block of $R_n$ is a basic block.  If $d$ is prime then every $G$-block of $R_n$ is a single branch.
\end{thm}

\begin{remark}
It follows that a $G$-block of $R_n$ that contains points lying in distinct major branches must have at least $2d^{n-1}$ elements, which settles Question \ref{fatou} for polynomials of prime-power degrees.
\end{remark}

\begin{proof}
Let $E$ be a $G$-block of $R_n$, and note that if $g \in G$ satisfies $g(E) \cap E \neq \emptyset$, then
$g(E) = E$, and it follows that the orbit under $g$ of any subset of $E$ is contained in $E$. By choosing an ordering of the vertices at level 1 of $T_n$, one obtains a natural representation of every element of $R_n$ as a base-$d$ string of length $n$.  For elements $v = v_1 \cdots v_n$ and $w = w_1 \cdots w_n$ of $R_n$, define the distance between $v$ and $w$ to be
$$D(v,w) =
\begin{cases}
n \qquad \qquad \text{if $v_1 \neq w_1$} \\
n - m \qquad \text{if $m$ is such that $v_{m+1} \neq w_{m+1}$ but $v_1 \cdots v_{m} = w_1 \cdots w_{m}$}
\end{cases}
$$
Choose $v, w \in E$ such that $D(v,w)$ is maximal.  Then $E$ is contained in a single branch of height $t := D(v,w)$. Let $s \in \N$ be such that $\sigma^s(v) = w$. Then $E$ contains the orbit of any element in $V$ under $\sigma^s$.  Moreover, $s$ is divisible by $d^{n-t}$ (since $\sigma^s$ fixes a branch of height $t$) but not by $d^{n-t+1}$ (since $\sigma^s$ does not fix branches of height $t-1$).  This gives
$d^{t-1} \leq | \sigma^s | \leq d^t$.  But since $d$ is a power of a prime it follows that $d^{t-1}$ divides the order of $\sigma^s$.  Thus a power of $\sigma^s$ is of order $d^{t-1}$ and induces a permutation of maximal order on each of the branches of height $t-1$. Since $E$ is invariant under the action of $\sigma$ it follows that $E$ is a union of branches of height $t-1$, which completes the proof.
\end{proof}

Before stating our next result, we recall a group $G$ acts {\em primitively} on a set $S$ when the only $G$-blocks of $S$ consist of either one point or all of $S$.  For instance, the symmetric group on $n$ letters acts primitively on $\{1, \ldots, n\}$.  Also, every doubly-transitive action is primitive.

\begin{thm} \label{main3}
Let $T_n$ be the complete $d$-ary rooted tree of height $n$, and let $R_n$ denote the vertices of height $n$.  Let $E \subseteq R_n$ be any set containing points in at least two major branches.  Suppose that $G \leq \Aut(T_n)$ contains a full $d^n$-cycle $\sigma$.  Suppose also that the restriction $G_1$ of $G$ to the height-1 vertices $R_1$ acts primitively. Then the smallest $G$-block containing $E$ is all of $R_n$.
\end{thm}

\begin{proof}  The idea of the proof is to use the presence of $\sigma$ to show that any $G_n$-block of $R_n$ restricts to a $G_1$-block of $R_1$, which by primitivity must be large.  Again using $\sigma$ one shows the only lift of such a block is all of $R_n$.

Let $A$ be any $G$-block of $R_n$, and denote by $A_1$ the restriction of $A$ to $R_1$.  Let $g_1 \in G_1$ satisfy $g_1(A_1) \cap A_1 \neq \emptyset$, implying that for some $\alpha, \beta \in A_1$ we have $g_1(\alpha) = \beta$.  Let $g \in G$ be any lifting of $g_1$, and let $B_\alpha, B_\beta$ be the major branches of $R_n$ corresponding to $\alpha, \beta$, whence $g(B_\alpha) = B_\beta$.  Note also that there are elements $a \in B_\alpha, b \in B_\beta$ belonging to $A$.
Now $\sigma^d$ maps each major branch to itself, and acts on each as a $d^{n-1}$-cycle.  Thus for some $m$ we have $\sigma^{dm} g(a) = b$.  Since $A$ is a block, it follows that $\sigma^{dm} g$ maps $A$ to itself.  But $\sigma^{dm}$ restricts to the identity on $R_1$, meaning that $\sigma^{dm} g$ restricts to $g_1$.  Thus $g_1$ maps $A_1$ to itself, proving that $A_1$ is a $G_1$-block.

Now let $E$ be as in the statement of the theorem, let $M$ be the minimal $G$-block of $R_n$ containing $E$, and let $M_1$ be the restriction of $M$ to $R_1$.  By the previous paragraph, $M_1$ is a $G_1$-block of $R_1$, and since $E$ contains points in different major branches, $\#M_1 \geq 2$.  By the primitivity of $G_1$, $M_1 = R_1$.  In particular, if $B$ and $B'$ are major branches of $R_n$ with
$\sigma(B) = B'$ then $M$ contains some $a \in B$ and some $b \in B'$.  Thus for some $m$,
$\sigma^{1+dm}(a) = b$.  Hence the orbit of $a$ under $\sigma^{1 + dm}$ is contained in $M$.  But since $1 + dm \equiv 1 \bmod{d}$, $\sigma^{1 + dm}$ acts on $R_n$ as a $d^n$-cycle.  Therefore
$M = R_n$.
\end{proof}

We note that if $G:=MG_n(f)$ is cyclic and $f$ has non-prime power order, which occurs for instance when $f$ is conjugate to $x^{pq}$ for primes $p < q$, then for every $n \geq 2$ there are $G$-blocks that are not basic.  Indeed, if $\sigma$ generates $G$ then an orbit of $\sigma^{q^n}$ is a $G$-block consisting of $p^n$ elements.  This cannot be a union of equal-height branches that is contained in a branch of height one more, since each such branch has a number of elements that is a power of $pq$. However, a polynomial for which the groups $G:=MG_n(f)$ are cyclic will have only a single critical point that is also a fixed point. For such a map it is easy to see that the answer to Question \ref{fatou} is positive, since all preimages of the point $p$ will lie in the same Fatou component. To give a negative answer we have to construct a more complicated map.

As noted above, polynomials conjugate to powers of $x$ have the strong property of having a single critical point that is also a fixed point.  As a first idea for making the map somewhat more complicated, one could simply require that all critical points of $f$ also be fixed points.  Such polynomials are called {\em conservative}, and they have been studied in some detail \cite{tischler, pakovich}.  In particular, their conjugacy classes are determined by relatively simple combinatorial data, and 
there is a natural faithful action of ${\rm Gal}(\overline{\Q}/\Q)$ on these conjugacy classes \cite{pakovich}. 
However, our next result shows that to answer Question \ref{fatou} in the negative, conservative polynomials will not do.  In Theorem \ref{main2} we give an example of a degree-6 polynomial that does provide a negative answer to Question \ref{fatou}

\begin{cor} \label{blockcor}
Suppose that $f \in \C[x]$ is conservative, and $f$ has at least two critical points.  Let $G = MG_n(f)$ act naturally on the tree $T_n$ of preimages of a suitable point $p$, and let $E \subset R_n$ contain points in at least two major branches of $R_n$.  Then the smallest $G$-block of $R_n$ containing $E$ is all of $R_n$.
\end{cor}

Corollary \ref{blockcor} is an immediate consequence of the following lemma.  Recall that a group $G$ acts doubly transitively on a set $S$ if for any $a,b,c,d \in S$ with $a \neq c$ and $b \neq d$, there is $g \in G$ with $g(a) = b$ and $g(c) = d$.  It is straightforward to show that a doubly-transitive action is primitive.

\begin{lem}
Suppose that $f \in \C[x]$ is conservative, and has at least two critical points.  Then $MG(f)$ acts
doubly-transitively on the set $R_1$ of preimages under $f$ of a generic point.
\end{lem}

\begin{proof}
It is enough to show that for $a, b, c \in R_1$ with $a \neq b, c \neq b$, there is $g \in MG(f)$ with $g(b) = b$ and $g(a) = c$.

By assumption the set of critical points of $f$ coincides with the set of critical values of $f$, which we write $\{v_1, \ldots, v_n\}$.  Each $v_i$ gives an element of $MG(f)$ that acts as a single $d_i$-cycle, where $d_i  = 1 + (\text{multiplicity of the critical value $v_i$})$.
Call this cycle $C_i$, and note that the $C_i$ generate $MG(f)$.   Note that
$$\sum_{i = 1}^n d_i = n + \sum_{i = 1}^n  \text{multiplicity of $v_i$} = n + d - 1.$$

Consider the graph $\Gamma$ whose vertex set consists of $R_1$, and where $u \neq v \in R_1$ are connected by one edge for each $C_i$ with $C_i(u) = v$, and also by one edge for each $C_i$ with $C_i(v) = u$.  There are no edges from any vertices to themselves.  Since the action of $MG(f)$ on $R_1$ must be transitive (e.g. since it contains a full cycle), $\Gamma$ is connected.  Hence one may think of the action of $MG(f)$ on $R_1$ as an interlinked system of circular conveyor belts, with the key property that each belt can be made to move independently of the others.

Now obtain a new graph $\Gamma'$ by deleting one edge from each cycle, which preserves connectedness.  The number of vertices of $\Gamma'$ is $d$, and the number of edges is $\sum_{i = 1}^n (d_i - 1)$, which is $n+d-1 - n = d-1$.  Denoting by $F$ the number of regions of the plane enclosed by the edges of the graph (including the region containing infinity), by Euler's formula we have $d - (d - 1) + F = 2$.  Thus $F = 1$, whence $\Gamma'$ is a tree.  This implies that each cycle can have at most one point in common with any other cycle.  Indeed, if $C_i, C_j$ share two distinct points $a$ and $b$, then $C_i^r(a) = b$ and $C_j^s(a) = b$ with $C_i^r$ and $C_j^s$ not the identity.  This implies that in $\Gamma$ there are four distinct paths from $a$ to $b$, namely those given by $C_i^r, C_i^{-r}, C_j^s, C_j^{-s}$.  Deleting any edge from $C_i$ and any from $C_j$ thus still leaves a cycle in $\Gamma'$.

Now let $a, b, c \in R_1$ be given, with $a \neq b, c \neq b$.  In the graph $\Gamma'$ mentioned above, consider a path from $a$ to $c$.  This path may be written as the action of
\begin{equation} \label{path}
C_{i_t}^{e_t} C_{i_{t-1}}^{e_{t-1}} \cdots C_{i_0}^{e_0}
\end{equation}
where the action is on the left, each $e_j$ is a nonzero integer, and $t \leq n$.  Choose the path to be non-self-intersecting, and so that $t$ is minimized.  This implies all the cycles in \eqref{path} are distinct, since a cycle $C$ occurring twice violates the minimality of $t$, as all cycles between occurrences of $C$ could be eliminated.  Since any two cycles meet in at most one point and the path given by \eqref{path} is non-self-intersecting, it follows that any element of $R_1$ may be moved by at most two of the cycles in \eqref{path}.  Moreover, these cycles must be consecutive.
If $b$ is left fixed by all the cycles in \eqref{path}, then we have moved $a$ to $c$ while fixing $b$ and are done.  If not, we let $k$ be maximal so that $C_{i_k}$ does not not fix $b$.  Suppose that $b$ is contained in two cycles, meaning it is the unique element moved by both $C_{i_k}$ and $C_{i_{k-1}}$
We claim the product
\begin{equation} \label{adjusted path}
C_{i_t}^{e_t} C_{i_{t-1}}^{e_{t-1}} \cdots C_{i_{k+1}}^{e_{k+1}} C_{i_{k-1}}^{-e_{k-1}} C_{i_k}^{e_k} C_{i_{k-1}}^{e_{k-1}}  \cdots C_{i_0}^{e_0}
\end{equation}
maps $a$ to $c$ and fixes $b$.    To show \eqref{adjusted path} maps $a$ to $c$, denote by $a'$ the image of $a$ under $C_{i_{k-2}}^{e_{k-2}} \cdots C_{i_0}^{e_0}$, and note it is enough to show
$$C_{i_{k-1}}^{-e_{k-1}} C_{i_k}^{e_k} C_{i_{k-1}}^{e_{k-1}} (a') = C_{i_k}^{e_k} C_{i_{k-1}}^{e_{k-1}}(a')$$
But $C_{i_k}^{e_k} C_{i_{k-1}}^{e_{k-1}} (a')$ is a point moved by $C_{i_k}$, and thus cannot be moved by $C_{i_{k-1}}$ unless it is $a'$, which violates the non-self-intersecting property of the path corresponding to \eqref{path}.  To show \eqref{adjusted path} fixes $b$, we need only show
\begin{equation} \label{partial path}
C_{i_{k-1}}^{-e_{k-1}} C_{i_k}^{e_k} C_{i_{k-1}}^{e_{k-1}}
\end{equation}
fixes $b$.  To see this, note that $C_{i_{k-1}}^{e_{k-1}}$ does not fix $b$, thus moving $b$ away from the unique element moved both by $C_{i_{k-1}}$ and $C_{i_{k}}$.  Hence $C_{i_k}^{e_k}$ fixes
$C_{i_{k-1}}^{e_{k-1}}(b)$, implying that \eqref{partial path} fixes $b$.

In the case where $b$ is contained in only one cycle $C_{i_k}$ of \eqref{path} and $k \neq 0$, one verifies in a similar manner that
$$C_{i_t}^{e_t} C_{i_{t-1}}^{e_{t-1}} \cdots C_{i_k}^{e_k} C_{i_{k-1}}^{e_{k-1}} C_{i_k}^{-e_k} C_{i_{k-2}}^{e_{k-2}}   \cdots C_{i_0}^{e_0}$$
moves $a$ to $c$ and fixes $b$.  To see that $b$ is fixed, let $\alpha$ denote the unique point contained in both $C_{i_k}$ and $C_{i_{k-1}}$, and note that
$C_{i_k}^{e_k}(\alpha)$ is either $c$ (if $k = t$)or the unique point contained in both $C_{i_k}$ and $C_{i_{k+1}}$.
In either case, $C_{i_k}^{-e_k}(b) \neq \alpha$, implying that $C_{i_k}^{-e_k}(b)$ is fixed by $C_{i_{k-1}}$.
If $k = 0$, one checks similarly that
$$C_{i_t}^{e_t} C_{i_{t-1}}^{e_{t-1}} \cdots C_{i_2}^{e_2} C_{i_0}^{-e_0} C_{i_{1}}^{e_{1}}  C_{i_0}^{e_0}$$
maps $a$ to $c$ and fixes $b$.

Finally, if \eqref{path} consists of only one cycle $C_0$, then choose a cycle $C'$ that moves a point
$\alpha$ also moved by $C_0$.  This is possible since $n \geq 2$ by hypothesis.  If $j$ is such that $C_0^j(b) = \alpha$, then the element
$$C_0^{-j} C'^{-1}C_0^{e_0} C' C_0^j$$
maps $a$ to $c$ and fixes $b$.
\end{proof}

\begin{thm}\label{main2}
There exists a polynomial $h$ of degree $6$ such that $R_2$ contains an $MG_2(h)$-block with only 4 elements, two each in two di?erent major branches of $R_2$.   Moreover, $h$ can be chosen so that it provides a negative answer to Question \ref{fatou}.
\end{thm}
\begin{proof}
The construction of the polynomial $h$ is very similar to the constructions in the proof of Theorem \ref{portrait2}.  We will assume that the reader is familiar with the polynomial $f$ from Theorem \ref{portrait2}.

Here the polynomial $h$ will again have critical points at $0 < c_1 < 1 < c_2 < 2$ and again $h(c_1) = h(c_2) = 2$ and $h(0) = h(2) = 0$. However this time $h(1)$ is not equal to $1$ but make sure that $0 < h(1) < c_1$ in such a way that $h^{\circ 2}(1) = 1$. Hence there are $4$ post-critical values, namely $0, a, 1$, and $2$, where $a = h(1)$, and the iterated monodromy group of $h$ has four generators.

\begin{figure}[htbp]
\begin{center}
\includegraphics{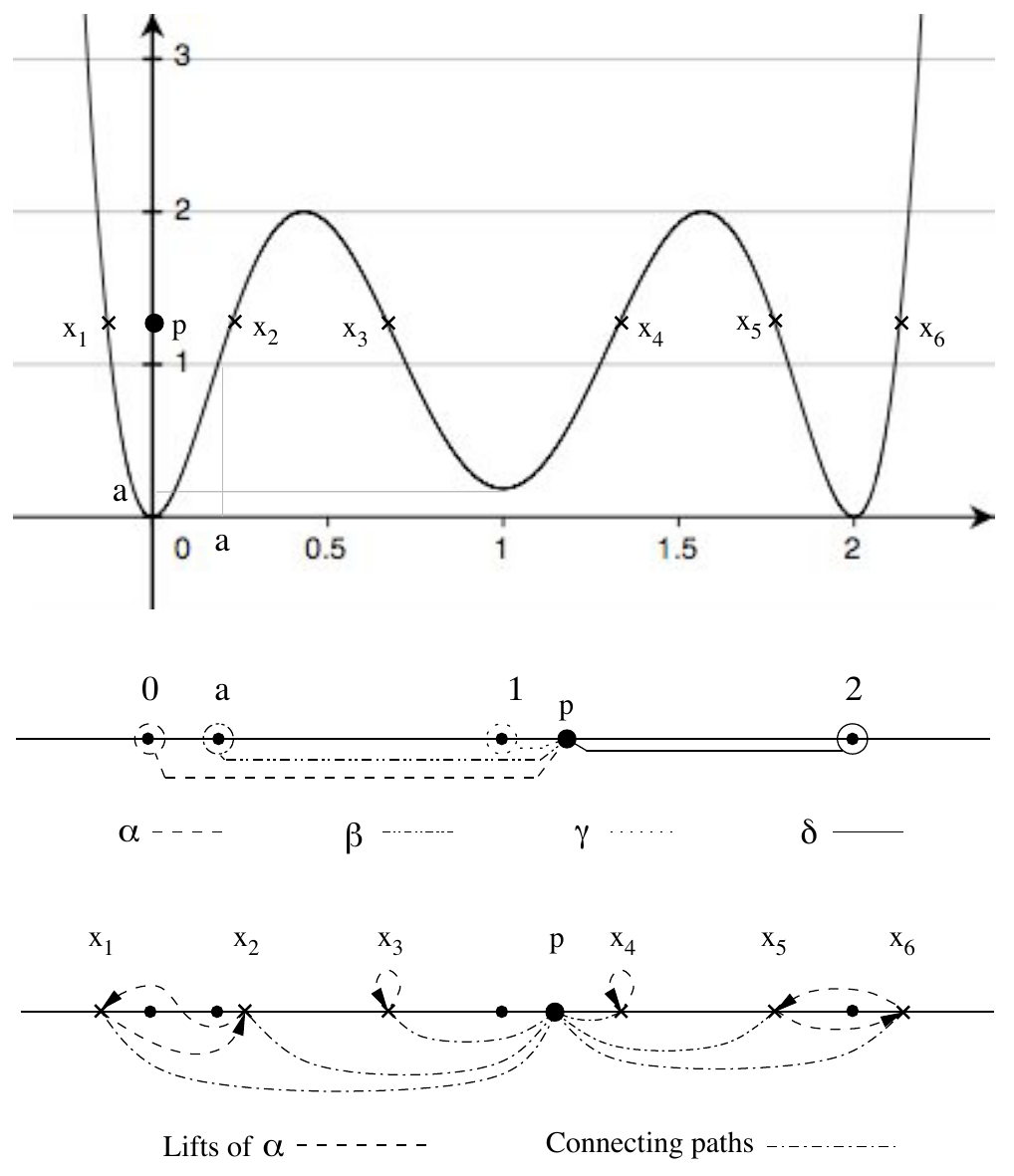}
\end{center}
\caption{Real graph of $h$, generating loops of the corresponding fundamental group
(all loops are counterclockwise), and computation of the action of $\alpha$ on $T$ (all
connecting paths move away from $p$).}
\label{fig: hgroup}
\end{figure}

We denote the generator loops encircling $0, a, 1$, and $2$, by $\alpha, \beta, \gamma,$ and $\delta$, respectively.  The loops are again defined by paths that stay very close to the real axis as in the proof of Theorem \ref{portrait2} (see Figure \ref{fig: hgroup}).  To determine the action of these generators on $T$, we use the method of \cite[Chapter 5]{nekrashevych}.  For instance, consider $\alpha$.  We fix paths $\pi_i$ connecting the basepoint $p$ to each $x_i$, and then for each $x_i$ we compute the lift $\ell_i$ of $\alpha$ beginning at $x_i$.  Then given a vertex of $T$ labeled with a word $iw$, $\alpha$ returns the word $j g(w)$, where $x_j$ is the endpoint of the lift of $\ell_i$ and $g$ is the element corresponding to the loop $\pi_i \ell \pi_j^{-1}$.  The bottom part of Figure \ref{fig: hgroup} illustrates this.  We thus see that the action of $\alpha$ on $T$ is given by
\begin{eqnarray*}
\alpha \cdot 1 & = & 2 \cdot {\rm id} \\
\alpha \cdot 2 & = & 1 \cdot \alpha \\
\alpha \cdot 3 & = & 3 \cdot {\rm id} \\
\alpha \cdot 4 & = & 4 \cdot {\rm id}\\
\alpha \cdot 5 & = & 6 \cdot {\rm id}\\
\alpha \cdot 6 & = & 5 \cdot \delta
\end{eqnarray*}
Hence in wreath recursion notation, we have
$\alpha = <{\rm id}, \alpha, {\rm id}, {\rm id}, {\rm id}, \delta> (1 \hspace{0.1 in} 2)(5 \hspace{0.1 in} 6)$.
In many cases one can read the action of an element of the fundamental group directly from the real graph.  For instance, to compute the action of $\gamma$, note that following preimages of $\gamma$ is the same as moving the $x_i$ down along the real graph of $h$ towards $1$.  Since none of the $x_i$ are near to each other when $y = 1$, the action of $\gamma$ on the first level of $T$ is trivial.  When the $y$-values of these points reaches $1 + \epsilon$, a counterclockwise loop is executed (in the complex plane, and so not visibly on the real graph).  One sees that since $x_2$ is near $a$, this loop must encircle $a$.  For the others, the loops encircle no critical values, and so are homotopically trivial.  Thus
$$\gamma  =  <{\rm id}, \beta, {\rm id}, {\rm id}, {\rm id}, {\rm id}>. $$
Similar computations yield
\begin{eqnarray*}
\beta & = & <{\rm id}, {\rm id}, {\rm id}, \gamma, {\rm id}, {\rm id}> (3 \hspace{0.1 in} 4) \\
\delta & = & <{\rm id}, {\rm id}, {\rm id}, {\rm id}, {\rm id}, {\rm id}> (2 \hspace{0.1 in} 3)(4 \hspace{0.1 in} 5)
\end{eqnarray*}
The action of $IMG(h)$ on the first two levels of $T$ is pictured in Figure \ref{fig: faction}.  Note that the action of $IMG(h)$ on the first level of the tree is the same as the action of the monodromy group of $f$ from Theorem 2.2. 
\begin{figure}[htbp]
\begin{center}
\includegraphics{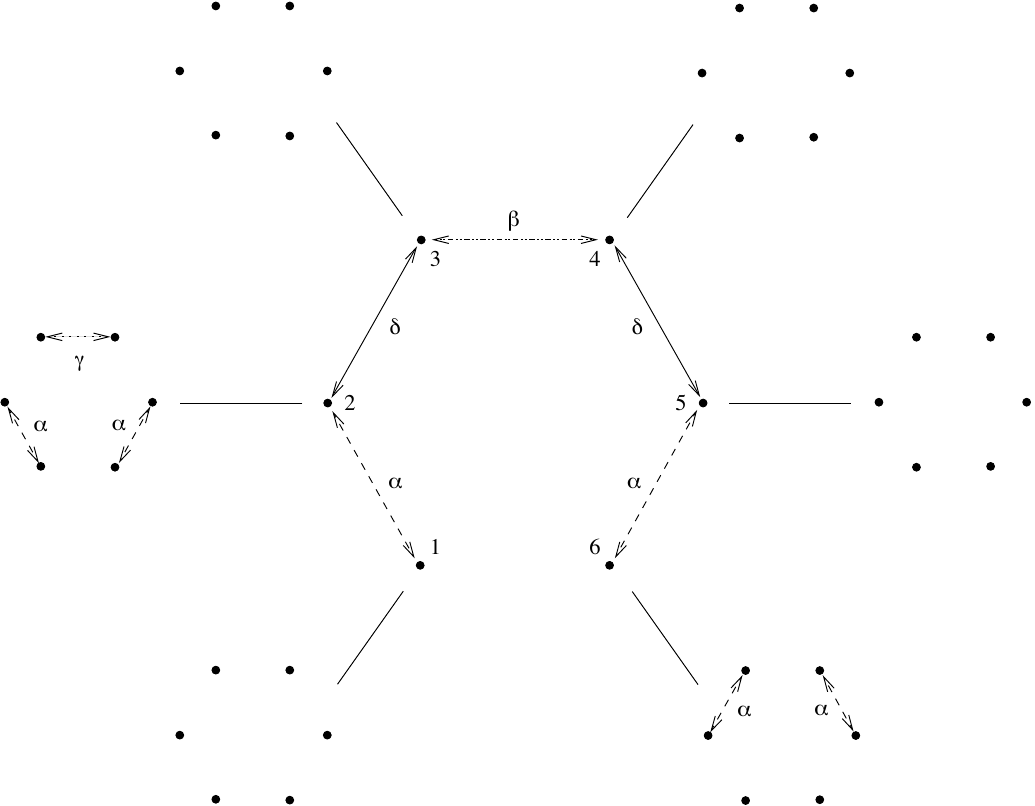}
\end{center}
\caption{Action of $IMG(h)$ on the first two levels of $T$.}
\label{fig: faction}
\end{figure}

Label a vertex on level two of $T$ by $x_{ij}$, where $i$ denotes the major branch and $j$ denotes the position within the $i$th major branch.  Consider the partition consisting of 
$$\{x_{ij}, x_{i(7-j)}, x_{(7-i)(4-j)}, x_{(7-i)(3+j)}\} \qquad i = 2,3,  j = 1,2,3$$
as well as $\{x_{11}, x_{16}, x_{62}, x_{65}\}, \{x_{12}, x_{15}, x_{63}, x_{64}\}$, and $\{x_{13}, x_{14}, x_{61}, x_{66}\}.$
One checks that all of $\alpha, \beta, \gamma, \delta$ respect this partition, and thus we have  a set of $G$-blocks each containing four elements.  Note that none of these blocks is basic.  

To show that $h$ provides a negative answer to question \ref{fatou}, note that $a$ is a super-attracting periodic point so it is contained in a Fatou component $V$ that is a super-attracting basin.  If $N$ is chosen odd then $p$ must be chosen in the Fatou set containing $1$, which does not contain a critical value. One easily sees that all preimages $h^{\circ - N}(p)$ in $V$ lie in the same major branch which is a $MG_N(h)$-block and contains $6^{N-1}$ elements.

If $N=2$ and thus $p \in V$ then the preimages $f^{\circ - 2}(p)$ that lie in $V$ correspond to the points $x_{32}, x_{42}$ which lie in the $MG_2(h)$-block $E = \{x_{32}, x_{35}, x_{42}, x_{45}\}$. So the smallest $MG_2(h)$-block containing those two points must be contained in $E$ (and is in fact equal to $E$). $E$ has only $4$ elements which is fewer than $6^1 + 1 = 7$. If $N$ is even but larger than $2$ then again $p$ must be chosen in $V$ and all preimages of $p$ in $V$ are also preimages of points in $E$. That gives at most $4\cdot 6^{N-2}$ elements and thus fewer than $6^{N-1}+1$. This completes the proof.

\end{proof}

\section*{acknowledgement}
The authors wish to thank Laurent Bartholdi for reading a preliminary version of the manuscript and for several helpful comments.

\bibliographystyle{plain}

\begin{thebibliography}{10}

\bibitem{bartholdi}
Laurent Bartholdi and B{\'a}lint Vir{\'a}g.
\newblock Amenability via random walks.
\newblock {\em Duke Math. J.}, 130(1):39--56, 2005.

\bibitem{bux}
Kai-Uwe Bux and Rodrigo P{\'e}rez.
\newblock On the growth of iterated monodromy groups.
\newblock In {\em Topological and asymptotic aspects of group theory}, volume
  394 of {\em Contemp. Math.}, pages 61--76. Amer. Math. Soc., Providence, RI,
  2006.

\bibitem{debes}
Pierre D{\`e}bes.
\newblock M\'ethodes topologiques et analytiques en th\'eorie inverse de
  {G}alois: th\'eor\`eme d'existence de {R}iemann.
\newblock In {\em Arithm\'etique de rev\^etements alg\'ebriques
  ({S}aint-\'{E}tienne, 2000)}, volume~5 of {\em S\'emin. Congr.}, pages
  27--41. Soc. Math. France, Paris, 2001.

\bibitem{kai}
Vadim~A. Kaimanovich.
\newblock ``{M}\"unchhausen trick'' and amenability of self-similar groups.
\newblock {\em Internat. J. Algebra Comput.}, 15(5-6):907--937, 2005.

\bibitem{nekrashevych}
Volodymyr Nekrashevych.
\newblock {\em Self-similar groups}, volume 117 of {\em Mathematical Surveys
  and Monographs}.
\newblock American Mathematical Society, Providence, RI, 2005.

\bibitem{nekcantor}
Volodymyr Nekrashevych.
\newblock A minimal {C}antor set in the space of 3-generated groups.
\newblock {\em Geom. Dedicata}, 124:153--190, 2007.

\bibitem{nekpoly}
Volodymyr Nekrashevych.
\newblock Combinatorics of polynomial iterations.
\newblock In {\em Complex dynamics}, pages 169--214. A K Peters, Wellesley, MA,
  2009.

\bibitem{pakovich}
Fedor Pakovich.
\newblock Conservative polynomials and yet another action of {${\rm
  Gal}(\overline{\Bbb Q}/\Bbb Q)$} on plane trees.
\newblock {\em J. Th\'eor. Nombres Bordeaux}, 20(1):205--218, 2008.

\bibitem{han}
Han Peters.
\newblock Constant weighted sums of iterates.
\newblock {\em Complex Var. Elliptic Equ.}, 54(3-4):371--386, 2009.

\bibitem{tischler}
David Tischler.
\newblock Critical points and values of complex polynomials.
\newblock {\em J. Complexity}, 5(4):438--456, 1989.

\bibitem{wielandt}
Helmut Wielandt.
\newblock {\em Finite permutation groups}.
\newblock Translated from the German by R. Bercov. Academic Press, New York,
  1964.

\end{thebibliography}

\end{document}